\documentclass[10pt,a4paper]{article}
\usepackage[utf8]{inputenc}
\usepackage{amsmath}
\usepackage{amsfonts}
\usepackage{amssymb}
\usepackage{tabularx,array,amsmath}
\usepackage{enumerate,graphicx}
\usepackage{amsfonts, amsthm}
\usepackage{amssymb}
\usepackage{multirow}
\usepackage{bigstrut}

\newtheorem{theorem}{Theorem}[section]
\newtheorem{lemma}[theorem]{Lemma}
\newtheorem{proposition}[theorem]{Proposition}
\newtheorem{corollary}[theorem]{Corollary}

\theoremstyle{definition}
\newtheorem{definition}[theorem]{Definition}

\newcommand{\V}{\mathcal{V}}

\newcommand{\T}{\mathcal{T}}

\newcommand{\modd}{\ \mbox{mod} \ }

\newcommand{\rep}{\zeta^{irr}_{{M_n},p}(s)}
\newcommand{\VVV}{\mathbf{v}}

\begin{document}

\title{Irreducible representations of a family of groups of maximal nilpotency class I: the non-exceptional case}
\author{Shannon Ezzat}

\maketitle

\begin{abstract}
We use a constructive method to obtain all but finitely many $p$-local representation zeta functions of a family $M_n$ of finitely generated nilpotent groups with maximal nilpotency class. For representation dimensions coprime to all primes $p<n$, we construct all irreducible representations of $M_n$ by defining a standard form for the matrices of these representations and, after taking into account twisting and isomorphism, count these twist isoclasses to obtain our $p$-local zeta functions.   
\end{abstract}

\section{Introduction}

Representation growth is a fairly recent area in group theory where one studies (usually infinite) groups by studying the sequence of the number of (sometimes equivalence classes of) irreducible complex representations of degree $n$ for all $n \in \mathbb{N}.$ More formally, for all $n \in \mathbb{N}$ let $r_n(G)$ be the number of irreducible representations of degree $n$. If all $r_n(G)$ are finite, we can study this sequence by embedding them as coefficients in a zeta function. 

In this paper, we deal with the case of a family of finitely generated torsion-free nilpotent groups. Henceforth, we call finitely generated torsion-free nilpotent groups $\T$-groups. If $G$ is a $\T$-group, it is known that all $r_n(G)$ are infinite. However, we can redefine $r_n(G)$ as the number of irreducible  representations ``up to twisting''.

More formally, let $G$ be a $\T$-group. Let $\chi$ be a 1-dimensional complex representation and $\rho$ an $n$-dimensional complex representation of $G$. We define the product $\chi  \otimes \rho$ to be a \emph{twist} of $\rho$. Two representations $\rho$ and $\rho_*$ are \emph{twist-equivalent} if for some 1-dimensional representation $\chi$, we have that $\chi \otimes \rho \cong \rho_*.$ This twist-equivalence is an equivalence relation on the set of irreducible representations of $G$. In \cite{LM} Lubotzky and Magid call the equivalence classes \emph{twist isoclasses}. We say $S_\rho,$ the twist isoclass containing and irreducible representation $\rho$, is of dimension $n$ if and only if $\rho$ is an $n$-dimensional representation. They also show that there are only finitely many irreducible $n$-dimensional complex representations up to twisting and that for each $n \in \mathbb{N}$ there is a finite quotient $G(n)$ of $G$ such that each $n$-dimensional irreducible representation $\rho$ of $G$ is twist-equivalent to one that factors through $G(n)$. Henceforth we call the $n$-dimensional complex representations of $G$ simply representations. We denote the number of twist isoclasses of irreducible representations of dimension $n$ by $r_n(G)$ or $r_n$ if no confusion will arise.

We now discuss embedding the sequence $(r_n)$ as coefficients of a zeta function. Consider the formal expression \begin{equation*} \zeta^{irr}_G(s) = \sum_{n=1}^\infty r_n(G) n^{-s}.
\end{equation*} If $\zeta^{irr}_G(s)$ converges for a right half plane of $\mathbb{C}$, say $D$, where $D := \{s \in \mathbb{C} \ | \ \Re(s) > \alpha\}$ for some $\alpha \in \mathbb{R}$, we call $\zeta^{irr}_G: D \to \mathbb{C}$ the (global) \emph{representation zeta function} of $G$. For any $\T$-group such a $D$ always exists \cite[Lemma 2.1]{SV}. We call $\alpha_G:=\alpha$ the \emph{abscissa of convergence} of $\zeta^{irr}_G(s)$. Let $\zeta^{irr}_{G,p} (s)$, where \begin{equation*}\zeta^{irr}_{G,p} (s) = \sum_{n=0}^\infty r_{p^n}(G) p^{-ns}, \end{equation*} be the \emph{$p$-local representation zeta functions} of $\zeta^{irr}_G(s).$ Considering the domain $D$ of $\zeta^{irr}_{G,p}(s)$ as above, we say $\alpha_{G,p}:=\alpha_p$ is the \emph{$p$-local abscissa of convergence} of $\zeta^{irr}_{G,p}(s).$ 

We know, by \cite[Theorem 6.6]{LM} that in each twist isoclass there exists a representation $\rho$ such that $\rho$ factors through a finite quotient. Since $G$ is nilpotent, its finite quotients are nilpotent and therefore decompose as a direct product of their Sylow-$p$ subgroups. Since the irreducible representations of direct products of finite groups are the tensor products of irreducible representations of their factors, its representation zeta function decomposes into an Eulerian product of its $p$-local representation zeta functions and therefore $\zeta^{irr}_G (s) = \prod_{p} \zeta^{irr}_{G,p} (s)$.  Moreover, it was shown by Hrushovski and Martin \cite{HM} that these $p$-local representation zeta functions are rational functions in $p^{-s}.$

Let $M_n = \langle a_1, \ldots, a_n,b  \ | \ [a_i,b]=a_{i+1} \rangle$. Following the conventions of other papers in this area, all commutators that do not appear in (or follow from) the relations are trivial. In this paper, we discuss the irreducible representations and the representation zeta functions of this family of groups. Additionally, we refer often to certain subgroups of $M_n.$ It is clear that for $ 2 \leq k < n$ the group $M_k$ is isomorphic to a subgroup of $M_n.$ With a slight abuse of notation we let the subgroup $M_k= \langle  a_{n-k+1}, a_{n-k+2}, \ldots, a_{n},b \rangle.$

\section{Related Results}

The idea of using zeta functions to study representation growth was introduced in \cite{Witten}, in which Witten studies compact Lie groups. Later, representation zeta functions were studied in \cite{LMar}, where Lubotzky and Martin use representation zeta functions to study arithmetic groups, and in \cite{JZ} where Jaikin develops a method for calculating the representation zeta functions of compact $p$-adic analytic groups with property FAb. This method uses Howe's work \cite{Howe} on the Kirillov orbit method and the concept of $p$-adic integration to calculate the zeta functions.

Representation zeta functions of $\T$-groups were first studied by Hrushovski and Martin in \cite{HM} using model-theoretic methods. The study of representation growth of $\T$-groups was expanded by Voll in \cite{voll1}. In that paper, Voll develops a method for calculating $p$-local representation zeta functions for a given $\T$-group $G.$ This method, like the method that appears in \cite{JZ}, involves Howe's work in \cite{Howe}. Note that a large part of the method that appears in \cite{voll1} is also used in that paper to study other types of algebraic growth, including subgroup and subring growth. Stasinski and Voll, in \cite{SV}, generalize Voll's work in \cite{voll1} to $\T$-groups coming from unipotent group schemes. The authors also generalize the functional equation that appears in \cite{voll1}.

Representation zeta functions have been used to study other classes of groups. We briefly mention some work done in these areas. In \cite{Voll3} Avni et al.\ study compact $p$-adic analytic groups and arithmetic groups. In \cite{Avnietal}, Avni et al.\ study representations of arithmetic lattices and prove a conjecture by Larsen and Lubotzky. In \cite{Avni} Avni shows that arithmetic groups have representation growth with rational abscissa of convergence. Bartholdi and de la Harpe, in \cite{BdlH}, study representation zeta functions of wreath products with finite groups. Craven, in \cite{Craven}, gives lower bounds for representation growth for profinite and pro-$p$ groups.

\section{Layout of Paper}

We employ a constructive method, first used in \cite{Ezzat}, to calculate the irreducible representations and representation zeta functions of this family of groups. The constructive method allows us to calculate the $p$-local representation zeta functions of $M_n$ when $p$ is an exceptional prime. Informally, we say a prime $p$ is exceptional if $p$ appears as a denominator in the equations that determine the possible eigenvalues of the linear operators of our representation. We will formally define these primes later.

Indeed, the family $M_n$ is a good choice of a family of $\T$-groups to study using the constructive method  for a number of reasons. The relatively simple eigenspace structure of the irreducible representations, due to the large abelian subgroup inside, allows us to construct all of the irreducible representations of each $M_n$ for all but finitely many prime-power dimensions. Additionally, for small $n$, it is tenable to construct all of the irreducible representations for the exceptional-prime-degree representations as well. This gives us all of the irreducible representation theory (and thus the global representation zeta functions) for $M_2,M_3,$ and $M_4$. Also, it gives us an infinite number of examples of exceptional $p$-local zeta functions of groups of "almost arbitrary" nilpotency class; that is, given a number $c$, there is a prime $q$ such that $q+1>c$ and the $q$-local representation zeta function of $M_{q+1}$ is known. Furthermore, together with the results of the calculations of the $2$-local, and $2$-local and $3$-local representation zeta functions of $M_3$ and $M_4$, respectively,  we have two examples of global representation zeta functions of $\T$-groups of nilpotency class greater than $2$. To the author's knowledge, these are the only two examples in the literature. For results involving $p$-local representation zeta functions of exceptional primes, see the sequel to this paper \cite{MaxClassII}.

This paper describes the construction of the irreducible representations of the maximal class group of nilpotency class $n$, denoted $M_n$. This is achieved by calculating the irreducible representations of $p$-power dimension. The large abelian subgroup allows us to simultaneously diagonalize all but one element of the images of the generators and reflected by this fact is the relatively simple eigenspace structure of the irreducible representations. We note that the calculation is uniform for most primes, in fact primes $p$ not less than the nilpotency class $n$; denominators that appear in the matrices of the representation are smaller than the prime considered and therefore behave as units mod $p.$ When the prime considered is smaller than $n$, the calculation loses its uniformity and the structure of the matrices of the representation differs from the non-exceptional cases; again, these will be dealt with in this paper's sequel.

\section{Preliminaries}

First, we introduce an important lemma that gives us much information about the eigenspace structure of representations of certain nicely behaved $\T$-groups, including $M_n$. Before this lemma, we give a definition regarding eigenspaces of a set of linear operators.

\begin{definition}
Let $\mathcal{L}$ be a set of linear operators of a vector space $V.$ If a subspace $W \subseteq V$ is an eigenspace of each $L \in \mathcal{L}$ then we say that $W$ is a \emph{mutual eigenspace} of $\mathcal{L}.$
\end{definition}

\begin{lemma}\label{yform}
For some set of $\alpha_{j,k} \in \mathbb{Z}$ let $G:=\langle a_1, \ldots, a_n,b_1, \ldots b_m \ | \ [a_i,b_j]=A_{i,j} \rangle,$ where $A_{i,j}:=\prod_{k=i+1}^{n}a_k^{\alpha_{j,k}},$ be a $\mathcal{T}$-group and let $\rho(G)$ be an irreducible $p^N$-dimensional representation of $G.$ Also, let $\rho(a_i)=x_i, \ $ $\rho(b_j)=y_j$, and $X_{i,j}:=\rho(A_{i,j})$ for all $i \leq n$ and $j \leq m.$ Define $X:=\{x_1,\ldots, x_n\},$ and $Y:=\{y_1 \ldots y_m\}$. Then the mutual eigenspaces of $X$ are one-dimensional and there are $p^N$ distinct mutual eigenspaces.
\end{lemma}

\begin{proof}
Let $\mathcal{E}:= \{E_1, \ldots, E_t\}$ be the set of mutual eigenspaces of $X.$ We will show that if $\mathbf{v}$ is an eigenvector of $E_{j_1}$ then for any $y \in Y$ the vector $y\mathbf{v}$ is an eigenvector of $E_{j_2}$ for some $j_1,j_2 \leq t.$ We then show that $Y$ acts transitively on $\mathcal{E}$  and that all $E_j \in \mathcal{E}$ are of the same dimension, in fact a $p$-power. Finally, we show that each $E_j$ is one-dimensional, and $t=p^N.$

It is clear, by definition, that $x_n$ commutes with all $y \in Y.$ Let $E \in \mathcal{E}$ and let $\mathbf{v} \in E.$ For a given $y_j\in Y,$ consider $x_ny_j\VVV.$ Since $x_n$ is central we have that
\begin{equation}
x_ny_j\VVV=\lambda_{n,j}y_j\VVV
\end{equation} for some $\lambda_{n,j} \in \mathbb{C}^*.$

 For all $k<n,$ let $\lambda_{k}$ be such that $\lambda_{k}\VVV=x_k\VVV$. Now, as an induction, we choose $i<n$ and assume that, for each $h>i,$ $y_j\VVV$ is an eigenvector of each $x_h,$ with eigenvalue $\lambda_{h,j}.$ Then
\begin{equation}
\lambda_{i-1} y_j\VVV=y_j\lambda_{i-1} \VVV=y_jx_{i-1}\VVV=X_{j,i-1} x_{i-1}y_j\VVV=x_{i-1}X_{j,i-1}y_j\VVV=x_{i-1}\lambda^\prime y_j\VVV
\end{equation}
for some $\lambda^\prime \in \mathbb{C}^*.$ Note that the third equality is by the group relations and the final equality is by the inductive hypothesis. Thus $y_j\VVV$ is an eigenvector of $x_{i-1}$ with eigenvalue $\lambda_{i-1}(\lambda^\prime)^{-1}.$ This induction tells us that for any mutual eigenvector $\VVV$ of $X$ that, for any $y \in Y,$ $y\VVV$ is also a mutual eigenvector.

Let $\mathbf{v_1,v_2} \in E.$ For some $x_i$ and $\lambda_1,\lambda_2 \in \mathbb{C}^*$ let
\begin{equation} x_iy_j\mathbf{v_1}=\lambda_1y_j\mathbf{v_1} \text{ and } x_iy_j\mathbf{v_2}=\lambda_2y_j\mathbf{v_2}.
\end{equation}
It is clear that $\mathbf{v_1+v_2} \in E.$ Now consider
\begin{equation}
x_i y_j(\mathbf{v_1+v_2})=x_i y_j\mathbf{v_1}+x_i y_j\mathbf{v_2}=\lambda_1y_j\mathbf{v_1}+\lambda_2y_j\mathbf{v_2}.
\end{equation}
Since $y_j(\VVV_1+\VVV_2)$ must be an eigenvector of $x_i$ we have that $\lambda_1=\lambda_2$ and $y_j\cdot E=E_j$ for some $E_j \in \mathcal{E}.$ Since $y_j$ is invertible it preserves dimension and $\text{dim}(E_{j_1})=\text{dim}(E_{j_2})$ for $j_1,j_2 \leq t.$ Also, since $y_j$ was arbitrary and $\rho$ is irreducible it must be that $\text{span}(\langle Y \rangle \cdot E)$ must be the entire space $\mathbb{C}^{p^N}$ and thus $Y$ must act transitively on $\mathcal{E}.$  It follows, by counting, that $t=p^r$ and, for $j \leq t,$ we have that $\text{dim}(E_j)=p^s$ for $r,s$ such that $r+s=N.$

Let $B= \langle b_1, \ldots, b_m \rangle.$ For an eigenspace $E \in \mathcal{E}$ let $S=\text{Stab}_B(E),$ and let $Y_*:=\rho(S).$  Let $W \subseteq E$ be a $S$-stable subspace. Let $H = \langle S,a_1 \ldots, a_n\rangle$ and let $\eta:H \to GL_{p^{s}}(E)$ be the restriction of $\rho$ to $H.$ Since $E$ is a mutual eigenspace of the each $a_i$ it is clear that $\eta(a_i)=k_iI$ for some scalars $k_i$ and thus the $\eta$-stable subspaces are the $S$-stable subspaces. Consider the $B$-orbit of $W$, say $O.$ Since $\rho$ is irreducible then $\text{dim}(O)=p^N$ and since $W$ is $S$-stable it must be that $W=E.$ Thus $E$ has no proper stable subspaces and $\eta$ is irreducible. By \cite{LM} we have that $\eta$ factors through a finite quotient up to twisting. By assumption and Schur's Lemma, since $S$ is abelian, each $y_*\in Y_*$ must be a scalar matrix. Thus, since $\eta$ is irreducible, $\text{dim}(E)=1.$ It then follows that all mutual eigenspaces of $X$ are 1-dimensional and, since $\langle Y \rangle$ acts transitively on $\mathcal{E},$ we have that $|\mathcal{E}|=p^N.$
\end{proof}

We now introduce some notation for complex roots of unity.

\begin{definition}\label{def:SpN}
Let $S_p^{\infty}$ be the all complex $p^\ell$th roots of unity for all $\ell \in \mathbb{N}$ and $S_p^k$ be the $p^k$th roots of unity (and note that $S_p^k\backslash S_p^{k-1}$ are the primitive $p^k$th roots of unity). Define $s:S^\infty_p \to \mathbb{N}$ such that  $s(\lambda)=k$ if and only if $\lambda \in S_p^k \backslash S_p^{k-1}.$ If $s(\lambda)=k$ we say that $\lambda$ has \emph{depth} $k.$
\end{definition}

The calculation of the irreducible representations of $M_n$ will involve generalizations of triangle numbers, namely $k$-simplex numbers. Let $T_0(0)=1, \  T_0(j)=1,$ and $T_j(0)=0$ for $j \in \mathbb{N}$ and recursively define $T_{k}(j)=\sum_{l=1}^j T_{k-1}(l)=T_{k}(j-1)+T_{k-1}(j)$  for $k \in \mathbb{N}.$ The next lemma lists some properties of these numbers that are needed. We state these without proof.

\begin{lemma} Let $i,j,k,b \in \mathbb{N}$ and $T_{k}(j)$ be defined as above.
\begin{enumerate}[i.] \label{tlemma}
\item $T_{k}(i)= {{i+k-1}\choose{k}}=\frac{i(i+1)\ldots(i+k-1)}{k!}.$
\item $T_{k}(i)-T_k(j) = (i-j)\frac{\gamma}{k!}$ for some $\gamma \in \mathbb{Z}.$
\item Let $p > k.$ Then for any $b \in \mathbb{N}$ and $\alpha$ such that $1 \leq \alpha \leq p-1$ we have $T_k(\alpha p^b + j) = T_k(j) \mod p^b.$
\item $T_k(j+1)=T_k(j)+T_{k-1}(j)+ \ldots + T_0 (j).$
\item $T_k(i+j) = \sum_{l=0}^{k}T_l(i)T_{k-l}(j).$
\item If $p>k$ then $T_k(p^N-1)=0 \mod p^N.$
\end{enumerate}
\end{lemma}

As a corollary of (iii) we have the following.

\begin{corollary}\label{termequality}
Let $p$ be a prime, let $k <p$, let $N \geq 1,$ let $1 \leq m \leq N,$ let $\alpha \in \mathbb{N}$ such that $p \nmid \alpha,$ and, for $j \geq 0$, let
    \begin{equation}
    \Gamma(k,j)= \alpha p^{m} T_k(j-1).
    \end{equation}
Then we have that $\Gamma(k,\beta p^{N-m}+j+1)=\Gamma(k,j+1) \mod p^N$ for all $\beta$ such that $1 \leq \beta < p^m$ and all $j$ such that $0 \leq j \leq p^{N-m}-1.$
\end{corollary}

\begin{proof}
Consider $\Gamma(k,\beta p^{N-m}+j+1).$ We have that
    \begin{align}
    \Gamma(k,\beta p^{N-m}+j+1) ={}& \alpha p^m \frac{(\beta p^{N-m}+j)\ldots (\beta p^{N-m}+(j+k-1))}{k!} \mod p^N
    \end{align}
By Lemma \ref{tlemma}(iii), and noting that $p^m(p^{N-m})=0 \mod p^N$ we have that only the term with no factor of $p^{N-m}$ survives mod $p^N;$ that is,
\begin{equation}
\alpha p^m(\beta p^{N-m}+j)\ldots (\beta p^{N-m}+(j+k-1))= \alpha p^m j(j+1)\ldots(j+k-1)
\end{equation}
and thus $\Gamma(k,\beta p^{N-m}+j+1)=\Gamma(k,j+1).$
\end{proof}

\section{Representation Structure and Standard Form}

We now calculate the $p$-local representation zeta function of each $M_n$, which we denote $\rep$, by explicitly constructing representatives of each twist isoclass. Let $\rho$ be a $p^N$-dimensional irreducible representation of $M_n$ and let $x_i=\rho(a_i)$ and $y=\rho(b)$.

In this section we will choose a basis for the image of $\rho$ such that $y$ is in the form of a $p^N$-cycle permutation matrix and such that each $x_i$ is diagonal with each diagonal entry in a certain form, discussed later in the section. It is not necessary to state a basis to understand the eigenspace structure of $\rho.$ However, as a canonical basis is easy to determine in this case, we appeal to a basis as an indexing device on the set of mutual eigenspaces of $\{x_1, \ldots, x_n\}$. 

We begin by considering twisting. Since $x_2, \ldots, x_n$ are commutators they are invariant under twisting. We can twist $y$ and $x_1$ by any complex number. We remind the reader that we can obtain every $p^N$-power irreducible representation of $M_n$ by twisting a representative $\rho$ from each twist isoclass.

Since all $x_i$ commute they are all simultaneously diagonalizable. By \cite[Theorem 6.6]{LM} all irreducible representations factor through a finite quotient (up to twist equivalence) and thus by Schur's lemma the central element $x_n$ is a scalar matrix.

By the group presentation of $M_n$ it is clear we can apply Lemma \ref{yform}. Let $X=\{x_1,\ldots, x_n\}.$ Then we know the mutual eigenspaces of $X$ are $1$-dimensional and that there are $p^N$ distinct mutual eigenspaces. Also, we have that $\langle y \rangle$ must permute the eigenspaces of $X$ transitively. Thus, $y$ must act as a $p^N$-cycle on the mutual eigenspaces of $X.$ We choose our basis, with basis vectors $\{e_1, \ldots, e_{p^N}\}$ such that the $x_i$ are diagonal and

\begin{equation}\label{formy}
y =\left( \begin{array} {cccc} 0 &  &  & k  \\ 1 & \ddots & &  \\  & \ddots
 & \ddots &  \\  &  & 1
 & 0 \end{array} \right)
\end{equation}
for some $k \in \mathbb{C}^*.$ By cofactor expansion of $y-\lambda I$, we have that the characteristic equation is $\lambda^{p^N}=k,$ and thus the eigenvalues of $y$ are all of the $p^N$th roots of $k.$ However, we have the freedom to twist $y$ by $k^{\frac{-1}{p^N}}$ and therefore we can ensure the eigenvalues of $y$ are all of the $p^N$th roots of unity. We can choose a representative of our twist isoclass such that $y$ is a true $p^N$-cycle permutation matrix under some choice of basis, or equivalently, we can ensure that $k=1.$

We set up some notation. Let $\lambda_{i,j}$ be the $j$th entry on the diagonal of $x_i$. We let $x_n =\lambda_n I$ and $\lambda_{i}=\lambda_{i,1}$. Note that $\lambda_{n,j}=\lambda_n$ for all $j.$ By twisting we can ensure that $\lambda_1=1$. It will be shown that the $\lambda_{i,j}$, and thus $\rho$, are determined by the $\lambda_\ell,$ for all $\ell$ such that $i \leq \ell \leq n.$

We now determine the structure of the matrices $x_i$ and the allowable values for the $\lambda_i.$ The next lemma is the base case for the inductive lemma following it. Although we could start the induction with $x_n$, this case is trivial. For purposes of elucidation, we start this induction with $x_{n-1}.$ Note that this lemma is true for all primes.

\begin{lemma}\label{lambdan}
The matrix $x_{n-1}$ has the form

\begin{equation}x_{n-1} =\left( \begin{array} {cccc} \lambda_{n-1} &  &  &   \\  & \lambda_n \lambda_{n-1} & &  \\  &
 & \ddots &  \\  &  &
 & \lambda_n^{p^N-1} \lambda_{n-1} \end{array} \right).\end{equation}
Moreover, for any prime $p$, we have that $\lambda_n$ is a $p^N$th root of unity; that is $s(\lambda_n) \leq N$.
\end{lemma}

\begin{proof}
Since by our group relations $[x_{n-1},y]=x_n = \lambda_n$ we have that $\lambda_{n-1,j+1}=\lambda_n\lambda_{n-1,j}$ for $j=1,\ldots,p^N-1$ and $\lambda_{n-1,1}=\lambda_{n}\lambda_{n-1,p^N}$. Combining these equations we have that $\lambda_{n-1}=\lambda_n^{p^N}\lambda_{n-1}$ and therefore $\lambda_n \in S_p^N.$
\end{proof}

We remind readers of Lemma \ref{tlemma} for properties of numbers $T_k(i).$ We define these numbers in the paragraph directly before Lemma \ref{tlemma}.

\begin{lemma}\label{formx}
For $ 1 \leq i \leq n-1$  we have that $\lambda_{i,j} =\prod_{k=i}^{n} \lambda_{k}^{T_{k-i}(j-1)}$ and thus the matrix $x_i$ has the structure
\begin{equation}x_{i} =\left( \begin{array} {cccc} \lambda_{i} &  &  &   \\  &  \prod_{k=i}^{n} \lambda_{k}^{T_{k-i}(1)} & &  \\  &
 & \ddots &  \\  &  &
 &  \prod_{k=i}^{n} \lambda_{k}^{T_{k-i}(p^N-1)} \end{array} \right).\end{equation}
Moreover we have that
\begin{equation}\label{cor1}
\lambda_i^{p^N} \prod_{k=i+1}^n \lambda_k^{T_{k-i}(p^N-1)} = 1.
\end{equation}
\end{lemma}

\begin{proof}
Assume

\begin{equation}\label{pqpq} x_{i} =\left( \begin{array} {cccc} \lambda_{i} &  &  &   \\  &   \prod_{k=i}^{n} \lambda_{k}^{T_{k-i}(1)} & &  \\  &
 & \ddots &  \\  &  &
 &  \prod_{k=i}^{n} \lambda_{k}^{T_{k-i}(p^N-1)} \end{array} \right)\end{equation} for some $i.$ By the group relation $[x_{i-1},y]=x_i$ we have, for some $j\leq p^N-1$, that \begin{equation}\lambda_{i-1,j+1} = \lambda_{i,j+1} \lambda_{i-1,j} = \left(\prod_{k=i}^{n} \lambda_{k}^{T_{k-i}(j) }\right)\lambda_{i-1,j}\end{equation} and \begin{equation}\lambda_{i-1,1} = \lambda_i \lambda_{i-1,p^N}.\end{equation} Combining the above equations for each $j,$ we have that \begin{equation}\lambda_{i-1,j+1} =\lambda_{i-1}\prod_{k=i}^{n} \lambda_{k}^{\sum_{l=1}^{j}T_{k-i}(l)}= \lambda_{i-1}\prod_{k=i}^{n}\lambda_{k}^{T_{k-i+1}(j)}\end{equation} and
\begin{equation}\lambda_{i-1,1} = \lambda_{i-1} \lambda_i^{p^N} \prod_{k=i+1}^n \lambda_k^{T_{k-i}(p^N-1)}.\end{equation}

\end{proof}

We have shown that, up to twisting and isomorphism, that any irreducible representation must be of the form given above. We give this a name.

\begin{definition}
The matrices $x_1, \ldots, x_n, y$ are in \emph{standard form} if the $x_i$ are in the form of Lemma \ref{formx} and $y$ is in the form of Equation \ref{formy}. We say $\rho$ is in \emph{standard form} if, under a chosen basis, the matrices $x_1, \ldots, x_n, y$ are in standard form.
\end{definition}

\section{Stable Subspaces}

In this section we determine possible stable subspaces of a representation $\rho.$ We show that if $\rho$ is not irreducible then it must have a certain proper stable subspace; we name this $V_{p^{N-1}}.$ Thus, to determine if $\rho$ is irreducible, we only need to check if $V_{p^{N-1}}$ is a stable subspace of $\rho.$ In this vein, let $V_{p^k}$ be the subspace spanned by $\langle y \rangle \cdot (e_1+e_{p^{k}+1}+ \ldots +e_{(p^{N-k}-1)p^{k}+1})$. Note two things: first, $V_{p^k}$ has dimension $p^k$; second, if $V_{p^{k}}$ is a stable subspace of $\rho$ then so is $V_{p^j}$ for $j \geq k.$

We define the $n$-tuple $\Lambda_n(k):=(\lambda_{1,k}, \ldots, \lambda_{n,k})$ where $k$ is considered mod $p^N.$

\begin{lemma}\label{tupleequality}
For any $k_1,k_2$ if $\Lambda_n(k_1)=\Lambda_n(k_2)$ then $\Lambda_n(k_1+1)=\Lambda_n(k_2+1).$
\end{lemma}

\begin{proof}
By Lemma \ref{formx} we have that $\lambda_i \in S_p^{\infty}$ for all $i$ and that $\lambda_{i,j+1}=\lambda_{i+1,j+1}\lambda_{i,j}$ for all $j.$ Consider $\Lambda_n(k_1+1).$ It is clear to see that $\lambda_{n,k_1+1}=\lambda_{n,k_2+1}$ since $x_n$ is central.
Now, as our inductive step,  choose $h$ such that $h\leq n-1$ and assume that for all $i>h$ we have that $\lambda_{i,k_1+1}=\lambda_{i,k_2+1}.$ Consider $\lambda_{h,k_1+1}.$ By Equation \ref{pqpq}  \begin{equation}\label{faza}
\lambda_{h,k_1+1}=\lambda_{h+1,k_1+1}\lambda_{h,k_1}.
\end{equation}
Our inductive hypothesis holds for the first factor of the right hand side of Equation \ref{faza} and the initial assumption holds for the second factor. Thus we have that

\begin{equation}\label{faza2}
\lambda_{h,k_1+1}=\lambda_{h+1,k_1+1}\lambda_{h,k_1}=\lambda_{h+1,k_2+1}\lambda_{h,k_2}=\lambda_{h,k_2+1}.
\end{equation}

\end{proof}

Since $\rho$ is of dimension $p^N,$ Lemma \ref{tupleequality} and elementary counting tells us that if, for some $\beta_*,j,$ and $k,$ $\Lambda_n(k)=\Lambda_n(\beta_* p^j+k)$ where $p \nmid \beta_*$ then $\Lambda_n(k)=\Lambda_n(\beta p^j+k)$ for all $\beta$ such that $0 \leq \beta \leq p^{N-j}-1.$ This can be seen since $\beta_*$ is a unit in the additive group $\mathbb{Z}/p^{N-j}\mathbb{Z}$ and thus $\beta_*$ generates all of $\mathbb{Z}/p^{N-j}\mathbb{Z}.$ This argument gives us the following corollary of Lemma \ref{tupleequality}.

\begin{corollary}\label{minimalV}
Let $j$ be the minimal power such that $\Lambda_n(k)=\Lambda_n(\beta p^j +k)$ for all $\beta$ such that $0 \leq \beta \leq p^{N-j}-1$ and for any $k.$ Then $V_{p^j}$ is a stable subspace of $\rho$ and $V_{p^{j-1}}$ is not stable.
\end{corollary}

We define notation to this effect. Let $H \leq M_n$ and let $\mathcal{V}(\rho|_H)$ be the the minimal stable subspace $V_{p^j},$ as in Corollary \ref{minimalV}, of $\rho|_H.$ We say that $\mathcal{V}(\rho)=\mathcal{V}\big(\rho(M_n)\big).$

We can, in fact, say more about this minimal subspace:

\begin{corollary}\label{minimalpower}
 The number $j$ is minimal such that $\Lambda_n(1)=\Lambda_n(p^j+1)$ if and only if $\V (\rho)=V_{p^j}.$
\end{corollary}

%
%

\begin{corollary}\label{inclusionV}
Let $\rho:M_n \to GL_{p^N}(\mathbb{C})$ be a representation. Then, for $k < n$ if $\V(\rho|_{M_k})=V_{p^{j}}$ then $\V(\rho)=V_{p^\ell}$ for some $\ell$ such that $\ell \geq k.$
\end{corollary}

We know that if $V_{p^k}$ is $\rho$-stable then so is $V_{p^j}$ for $j \geq k.$ Thus, we obtain the following corollary:
\begin{corollary}\label{VpN-1}
Let $\rho$ be a representation of $M_n.$ The representation $\rho$ is irreducible if and only if $V_{p^{N-1}}$ is not $\rho$-stable.
\end{corollary}
Throughout this paper we use Corollary \ref{VpN-1} to check if a representation $\rho$ is irreducible. We use Corollary \ref{minimalV} to determine the number of isomorphic representations in standard form in one twist isoclass. Corollaries  \ref{minimalpower} and \ref{inclusionV} are used in the sequel to this paper, but are included here since they follow from Lemma \ref{tupleequality}.

\section{Isomorphic Representations in Standard Form}

Since representations in the same twist isoclass are equivalent under both twisting and isomorphism, we determine when two representations in standard form are isomorphic. In this vein, we have the following proposition.

\begin{proposition}
Let $\rho_1,\rho_2$ be irreducible representations of $M_n$ in standard form. Then $\rho_1$ and $\rho_2$ are in the same twist isoclass if and only if there is a 1-dimensional representation $\chi$ and a permutation matrix $P \in GL_{p^N}(\mathbb{C})$ such that $\rho_1=P\chi\rho_2 P^{-1}.$
\end{proposition}

Since one direction is immediate, we prove the other direction with the following lemma.

\begin{lemma}\label{preshout}

For any prime $p$, let $\rho: M_n \to GL_{p^N}(\mathbb{C})$ be irreducible and let $P$ be a matrix such that, for $1 \leq i \leq n$, the matrix $Px_iP^{-1}$ is diagonal and $PyP^{-1}=y.$ Then $P= T y^m$ for some $0 \leq m \leq p^N-1$ and scalar $T$. Furthermore, up to twisting, $Px_iP^{-1}$ and $PyP^{-1}$ are in standard form.
\end{lemma}

\begin{proof}
Let $X = \{x_1, \ldots, x_n\}.$ We will show that since all elements of $X$ are diagonal with $1$-dimensional mutual eigenspaces, $P$ must be a generalized permutation matrix. Then we show that since $P$ commutes with $y$ that $P$ must be a power of $y$ up to scalars. We then show that it follows that, up to twisting, $P\rho P^{-1}$ is in standard form.

Since all elements of $X$ are diagonal and its mutual eigenspaces are 1-dimensional, $C_{GL_{p^N}(\mathbb{C})}(\langle X \rangle)=D,$ where $C_G(H)$ is the centralizer of $H$ in $G$ and $D \leq GL_{p^N}(\mathbb{C})$ are the diagonal matrices. Since $D$ is the centralizer of $X$ and since $X$ has 1-dimensional mutual eigenspaces $D$ is also the centralizer of $PXP^{-1}.$ Let $N_G(H)$ be the normalizer of $H$ in $G.$ It is well known that $N_{GL_{p^N}(\mathbb{C})}(D)=\mathcal{P}$ where $\mathcal{P}$ are the generalized permutation matrices; that is, matrices with precisely one non-zero entry in each row and column. And since the mutual eigenspaces of $X$ are all 1-dimensional we have that $P\in \mathcal{P}.$

By definition of centralizer, it must be that $P$ must be in the centralizer of $y.$ Since $y$ is a $p^N$-cycle we have that $P=Ty^m$ for some diagonal matrix $T.$ But since $P$ commutes with $y$ and, of course, $y^m$ commutes with $y,$ $T$ must as well. It follows that $T$ must be a scalar matrix.

Conjugation of each $x_i$ by $P,$ which is the same as conjugating by $y^\ell$, for each $i$ and some $\ell$, maps $\lambda_{i}$ to $\lambda_{i,\ell+1}.$ We can twist by some 1-dimensional representation $\chi$ such that $\lambda_{1,\ell+1}=1$ and thus by, Proposition \ref{form} (or directly by Lemma \ref{tlemma}(v)), $\chi P\rho P^{-1}$ is in standard form.

\end{proof}

This ends the proof of the proposition.\\

Remembering that representations in a twist isoclass are equivalent up to both twisting and isomorphism, we make the following definition.

\begin{definition}
Let $\rho$ be irreducible and let $x_i,y,$  for $i$ such that $1 \leq i \leq n,$ be in standard form as defined earlier in the section. A \emph{shout} is a matrix $P$ such that, up to twisting, $PyP^{-1}$ and $Px_iP^{-1}$ for $i=1 \ldots n$ are in standard form. The representations $\rho$ and $P\rho P^{-1}$ (note that $Px_1P^{-1}$ may not be in standard form) are said to be equivalent under \emph{shouting}.
\end{definition}

We now need to count how many representations in standard form are in the same twist isoclass as $\rho;$ that is, how many representations in standard form are twist-equivalent to $\rho.$ We say that two representations that satisfy these conditions are equivalent under \emph{twisting and shouting} \cite{Twist}. If there are $d$ twist-and-shout equivalent representations,  and if we are just counting representations in standard form then we have overcounted by a factor of $d.$ Thus, we must take this into account when counting twist isoclasses. In this vein, we now have the following lemma.

\begin{lemma}\label{shout}
Let $S_\rho$ be the twist isoclass represented by $\rho$ and let $\V(\rho|_{M_{n-1}})=V_{p^{m}}.$ Then there are $p^m$ representations in standard form in $S_\rho$ that are twist-and-shout equivalent to $\rho.$
\end{lemma}

\begin{proof}
By Lemma \ref{formx} the entries of the $x_i$ are determined by the $\lambda_i.$ So to determine how many representations are twist-and-shout equivalent to $\rho$ we must count the number of choices of $\lambda_i$ such that $\rho_*=\chi P \rho P^{-1}$ such that $\rho_*$ is in standard form for some $1$-dimensional representation $\chi.$

Let $x_i^\prime=\chi P x_i P^{-1}$ for all $i\leq N$ and let $\lambda_i^\prime$ be the first diagonal entry of $x_i^\prime.$ By Lemma \ref{preshout} we have that $y=p^\ell$ and thus for some $\ell\leq p^N$ we have that $\lambda_i^{\prime}=\lambda_{i,\ell}$ for each $i.$ Since $\rho_*$ is in standard form it must be that we chose $\chi$ such that $\lambda_1^\prime=1.$

By the argument above, our choice of $\ell$ gives us, up to our choice of twist $\chi$ a representation that is twist-and-shout equivalent to $\rho.$ It follows that the number of representations twist-and-shout equivalent to $\rho$ is the size of the set $\{\Lambda_n^\prime(\ell):=(\lambda_{2,\ell},\ldots,\lambda_{n,\ell}) \ | \ \ell \leq p^N\}.$ By Corollary \ref{minimalV} we have that the size of this set is $p^m.$
\end{proof}

Note two things: first, that when we reference this lemma, we say we take shouting into account; and second, since all entries of any $x_i$ differ by products of $\lambda_j$ such that $j>i,$ this lemma implies that the depth of $\lambda_2$ has no effect on the number of twist-and-shout equivalent representations.

During the calculation of the $p$-local zeta functions that appear in this section, we break computation into various cases that depend on the depths of the $\lambda_i$ that we choose. We note, without additional special mention, that each case is closed under shouting. For completeness, however, we have the following lemma which can be applied to the various cases to show that they are closed.

\begin{lemma}
For some $i$ and $M$ let $s(\lambda_i)=M, s(\lambda_{i+1}), \ldots, s(\lambda_n) \leq M.$  Then, for all $j,$ $s(\lambda_{i,j}) \leq s(\lambda_{i}).$
\end{lemma}

\begin{proof}
Each $\lambda_{i,k}=\lambda_i \Lambda$ where $\Lambda$ is some product of the roots of unity $\lambda_{i+1}, \ldots, \lambda_{n}.$ Also, for $\lambda_a,\lambda_b \in S_p^\infty,$ we have that $s(\lambda_a \lambda_b) \leq \max\{s(\lambda_a),s(\lambda_b) \}.$ The result follows immediately from these two facts.
\end{proof}

\section{Irreducible Representations for Non-Exceptional Primes}

We note that the expressions $T_k(i)$ contain a denominator of $k!$. By Lemma \ref{formx} the maximal value for $k$ is $n-1$.  We say a prime $p$ is \emph{exceptional} if $p\leq n$ and \emph{non-exceptional} otherwise. For all non exceptional primes the $p$-local zeta function will behave uniformly. 

In this section we study the conditions for irreducibility of a $p^N$-dimensional representation $\rho$ such that $p \geq n.$ If this is the case, then the $T_k(i)$ terms that appear in the calculation of the standard forms have denominators that are all units mod $p.$ We show that such a representation $\rho$ is irreducible precisely when at least one of the $\lambda_i$ is a primitive $p^N$th root of unity.

We must determine the possible values for the $\lambda_i.$ By Equation \ref{cor1} in Lemma \ref{formx} and Lemma \ref{tlemma}(vi) we have the following lemma.

\begin{lemma}\label{depth}
For non-exceptional primes $p \geq n $ we have, for all $i \leq n$, that $\lambda_i \in S_p^N.$
\end{lemma}

Let $x_1,\ldots, x_n,y$ be matrices in standard form and let $\rho$ be the corresponding representation. We will show that, for non-exceptional primes, $\rho$ is irreducible precisely when one of the $x_i$ has all $p^N$th roots of unity on its diagonal. This implies that, at least one of the $\lambda_i$, where  $i \neq 1,$ is in fact a primitive $p^N$th root of unity. This will be shown in two stages. First, if $s(\lambda_i)=N$ for some $i \neq 1$ and $s(\lambda_{k}) \leq N-1$  for all $k$ such that $i+1 \leq k \leq n$ then $x_{i-1}$ has all $p^N$th roots of unity on its diagonal. Secondly, we show a stronger result that implies that if none of the $\lambda_i$ are primitive $p^N$th roots of unity, that is $s(\lambda_i) \leq N-1$ for all $i,$ then there is a proper stable subspace. We use the full strength of the second lemma in the sequel to this paper.

We state the above as a proposition. The proof is a consequence of the two lemmas following it.
\begin{proposition}\label{form}
Let $p\geq n$ and $\rho$ be a $p^N$-dimensional representation of $M_n$ with corresponding matrices in standard form. Then $\rho$ is irreducible if and only if there exists a $\lambda_i$  such that $s(\lambda_i)=N,$ where $2 \leq i \leq n.$
\end{proposition}

\begin{lemma}
If $s(\lambda_i)=N$ and $s(\lambda_k) \leq N-1$ for all $k$ such that $i+1 \leq k \leq n$ then all $\lambda_{i-1,j},$ where $\ 1\leq j \leq p^N$, are distinct $p^N$th roots of unity.
\end{lemma}

\begin{proof}

Assume that $s(\lambda_i)=N $ and  $s(\lambda_{k}) \leq N-1$ for $i+1 \leq k \leq n.$ We can write these non-primitive $\lambda_k$, for each $k$, as powers of $\lambda_i$. Let $\lambda_k=\lambda_i^{\alpha_k p^{m_k}}$ such that $p \nmid \alpha_k$ and $m_k \geq 1.$  For ease of display let $\alpha_i=1$ and $m_i=0$. Let $A_j= \sum_{k=i}^{n} \alpha_k p^{m_k} T_{k-i+1} (j-1).$ Then, by Lemma \ref{formx},
\begin{equation}
\lambda_{i-1,j}= \prod_{k=i-1}^{n} \lambda_{k}^{T_{k-i+1}(j-1)}=\lambda_{i-1}\lambda_i^{A_j}.
\end{equation}
We show that each diagonal entry is distinct by dividing two of them, say $\lambda_{i-1,s+1}$ and $\lambda_{i-1,t+1}$ with $s\geq t$, and showing that if $\lambda_{i-1,s+1}/\lambda_{i-1,t+1}=1$ then $s=t$.

Consider the equation
\begin{equation} \frac{\lambda_{i-1,s+1}}{\lambda_{i-1,t+1}} = \lambda_i^{A_{s+1}-A_{t+1}}=1. \end{equation} Taking the logarithm base $\lambda_i$ and working mod $p^N$:
\begin{align}\label{itrit}
 0 \modd p^N ={}& A_{s+1}-A_{t+1}\\
             ={}& \sum_{k=i}^{n} \alpha_k p^{m_k} T_{k-i+1} (s)- \sum_{k=i}^{n} \alpha_k p^{m_k} T_{k-i+1} (t)\nonumber \\
                                              ={}& (s-t) + [ \alpha_{i+1} p^{m_{i+1}} \big(T_2 (s) - T_2 (t)\big) + \ldots \nonumber\\
                        &{}+ \alpha_{n} p^{m_{n}} \big(T_{n-i+1} (s) - T_{n-i+1} (t)\big)].\nonumber
\end{align}

By Lemma \ref{tlemma}(ii) we have that $(s-t)$ is indeed a factor of each numerator of the right hand side of Equation \ref{itrit}. Therefore, remembering that $p \geq n$ and noting that all denominators are units mod $p$,

\begin{align}
 0 \modd p^N & = A_{s+1}-A_{t+1}\\
             & = (s-t)[1+ \alpha_{i+1} p^{m_{i+1}}\frac{\gamma_{i+1}}{2!}+ \ldots +\alpha_n p^{m_n}\frac{\gamma_n}{(n-i+1)!}] \nonumber
\end{align}

\noindent for some $\gamma_k$. Since all $m_k\geq 1$ this implies that

\begin{equation}
1+ \alpha_{i+1} p^{m_{i+1}}\frac{\gamma_{i+1}}{2!}+ \ldots +\alpha_n p^{m_n}\frac{\gamma_n}{(n-i+1)!} \neq 0 \modd p
\end{equation} and thus $s-t = 0 \mod p^N$ and we conclude that $s=t$. We can now say that each diagonal entry of $x_{i-1}$ is distinct.
\end{proof}

We prove the necessity of having at least one $\lambda_i$ a primitive $p^N$ root of unity in the following lemma. The idea is to show that if no $\lambda_i$ is a primitive $p^N$th roots of unity then for $1 \leq \beta  \leq p^{k}-1$ and $1 \leq j \leq p^{N-k},$ for some $k$, we have that $\lambda_{i,j}=\lambda_{i,\beta p^{N-k} + j}$ for any $i.$ Therefore $V_{p^k}$ is in fact a proper $\rho$-stable subspace of $\mathbb{C}^{p^N}$. For $\rho$ to be irreducible this cannot be the case.

\begin{lemma}\label{pNth}
Let $\lambda_* \in S_p^{N} \backslash S_p^{N-1}.$ For each $i \geq 2$ let  $\lambda_i=\lambda_*^{\alpha_i p^{m_i}}$ where $p \nmid \alpha_i$ and $m_i \geq 1.$ Also, let $m_*= \text{min}\{m_i\}.$  Then, for any $i$, $\lambda_{i,j}=\lambda_{i,\beta p^{N-m_*}+j}$ for all $1 \leq \beta \leq p^{m_*}-1$ and $1 \leq j \leq p^{N-m_*}.$
\end{lemma}

\begin{proof}
Consider the expression $L:= \text{log}_{\lambda_*} (\lambda_{i,\beta p^{n-m_*}+j+1})$ for some $0 \leq j \leq p^{n-m_*}-1$ where $1 \leq \beta \leq p^{m_*}.$ Then

\begin{align}\label{sizeofsubspace}
   L \mod p^N ={}& \alpha_i p^{m_i} + \alpha_{i+1} p^{m_{i+1}}T_1(\beta p^{N-m_*}+j) + \ldots \\
    &+ \alpha_{n} p^{m_n}T_{n-i}(\beta p^{N-m_*}+j) \nonumber\\
    ={}&\alpha_i p^{m_i} + \alpha_{i+1} p^{m_{i+1}}(\beta p^{N-m_*}+j) + \ldots\nonumber\\
    &+ \alpha_{k} p^{m_k} \frac{(\beta p^{N-m_*}+j) \ldots (\beta p^{N-m_*}+(j+k-i-1))}{(k-i)!} + \ldots\nonumber\\
    &+ \alpha_{n} p^{m_n} \frac{(\beta p^{N-m_*}+j) \ldots (\beta p^{N-m_*}+(j+n-i-1))}{(n-i)!} \nonumber
\end{align} where the term involving $k$ is a typical term. By Corollary \ref{termequality}, since $m_* \leq m_i$ for all $m_i$ it follows that, if each numerator was expanded, all terms are 0 mod $p^N$ but the terms that have no factor of $p^{N-m_*}$ in the expansion of each numerator; that is

\begin{equation}\alpha_{k}p^{m_{k}}(\beta p^{N-m_*}+j) \ldots (\beta p^{N-m_*}+j+k -1) = \alpha_{k}p^{m_{k}}(j) \ldots (j+k -1) \mod p^N
\end{equation}
for $i \leq k \leq n.$ Therefore we have that

\begin{align}
   L ={}& \alpha_i p^{m_i} + \alpha_{i+1} p^{m_{i+1}}T_1(j) + \ldots + \alpha_{n} p^{m_n}T_{n-i}(j) \modd p^N\\
    ={}& \text{log}_{\lambda_*} (\lambda_{i,j+1}).\nonumber
\end{align}
\end{proof}

This completes the proof of the proposition.\\

\section{Counting Non-Exceptional Twist-Isoclasses}

Now that we have determined all irreducible representations up to twisting and isomorphism for the non-exceptional case, we count the number of twist isoclasses. We do this by counting the number of $\rho$ that have a basis such that they are in standard form, and avoid overcounting by taking into account representations that are isomorphic under twisting and shouting. We remind the reader that, by the previous proposition, one of the $\lambda_i$, for $i \geq 2$, must be a primitive $p^N$th root of unity.

Regarding twist-and-shout equivalent representations, we have the following lemma that follows directly from Lemmas \ref{shout} and \ref{pNth}; notice the choices of $p$ that are valid for this lemma.

\begin{lemma}\label{nonexshout}
For $p\geq n-1$ let $\rho$ be an irreducible $p^N$-dimensional representation of $M_n$ and let $\V(\rho|_{M_{n-1}})=V_{p^k}.$ Then there are $p^k$ representations in standard form equivalent to $\rho$ under twisting and shouting.
\end{lemma}

We break the computation into two cases. First, assume $s(\lambda_k)=N$ for some $k$ such that $3 \leq k \leq n.$ In this case there are altogether $(1-p^{-(n-2)})p^{(n-2)N}$ choices for $\lambda_3, \ldots, \lambda_n$. We can choose any $p^N$th root of unity for $\lambda_2$ and therefore there are $p^N$ choices for this. By Lemma \ref{nonexshout} we must divide by $p^N$ to take shouting into account.

Now assume $s(\lambda_2)=N$ and $s(\lambda_i) \leq N-1$ for $3 \leq i \leq n$. There are $(1-p^{-1})p^N$ choices for $\lambda_2$. If $\text{max}\{s(\lambda_k)\} = \ell \neq 0$ for $3 \leq k \leq n$ then we have $(1-p^{-(n-2)})p^{(n-2)(N-\ell)}$ choices for these. By Lemma \ref{nonexshout} we are overcounting by a factor of $p^{N-\ell}$. If $\text{max}\{s(\lambda_k)\}=0$ for $3 \leq k \leq n$ then all of the $\lambda_3, \dots, \lambda_{n}$ are the $p^0$th root of unity, namely $1$. Since we have no freedom to shout in this case, we are not overcounting.

Summing these two cases together we have, for $N \geq 1,$
\begin{align}\label{rnonex}
r_{p^N} ={}& (1-p^{-(n-2)})p^{(n-2)N}p^N p^{-N} \\
&+  \sum_{l=1}^{N} (1-p^{-1})p^N (1-p^{-(n-2)})p^{(n-2)(N-l)} p^{-(N-l)}\nonumber \\
& +  (1-p^{-1}) p^N \nonumber
\end{align}
and
\begin{align}\label{nonex}
 \rep ={}& \sum_{N=0}^{\infty} r_{p^N}p^{-Ns} = 1 + \sum_{N=1}^{\infty} (1-p^{-(n-2)})p^{(n-2)N}p^N p^{-N} p^{-Ns}\\
&+ \sum_{N=1}^{\infty} \sum_{\ell=1}^{N-1} (1-p^{-1})p^N (1-p^{-(n-2)})p^{(n-2)(N-\ell)} p^{-(N-\ell)} p^{-Ns}\nonumber\\
& + \sum_{N=1}^{\infty} (1-p^{-1}) p^N p^{-Ns}\nonumber\\
\end{align}
Summing the geometric series we have
\begin{align}
\rep ={}& 1+(1-p^{-(n-2)})\frac{p^{(n-2)-s}}{1-p^{(n-2)-s}}\label{xeaas}\\
&+\frac{(1-p^{-1})(1-p^{-(n-2)})}{1-p^{3-n}}\left( \frac{p^{1-s}}{1-p^{(n-2)-s}}-\frac{p^{1-s}}{1-p^{1-s}}\right)\nonumber\\
&+ (1-p^{-1})\frac{p^{1-s}}{1-p^{1-s}}\nonumber
\end{align} and a routine calculation of Equation \ref{xeaas} yields that
\begin{equation}\label{znonex}
\rep = \frac{(1-p^{-s})^2}{(1-p^{(n-2)-s})(1-p^{1-s})}.
\end{equation} Note in particular that $\rep\mid_{p\rightarrow p^{-1}} = p^{n-1} \rep$ and thus this zeta function does indeed satisfy the correct functional equation in \cite{voll1}. By Equation \ref{znonex} we can also say that the $p$-local abscissa of convergence is
\begin{equation}
\alpha_{M_n,p}=n-2
\end{equation} for $n \geq 3.$ If $n=2$ then a factor of $(1-p^{-s})$ in the numerator cancels with the factor $(1-p^{(n-2)s})$ in the denominator. It follows that
\begin{equation}
\alpha_{M_2,p}=1.
\end{equation}

\bibliographystyle{plain}
\bibliography{MaxClassPaperReferences}{}

\end{document}